\documentclass[a4paper,11pt]{article}

\usepackage[top=24mm,bottom=25mm,left=25mm,right=25mm]{geometry}
\usepackage[T1]{fontenc}
\usepackage{lmodern}
\usepackage{amsmath,amssymb,amsthm,mathtools}
\usepackage{bm}
\usepackage{booktabs}
\usepackage{enumitem}
\usepackage{titlesec}
\usepackage[hidelinks]{hyperref}
\usepackage{url}

\allowdisplaybreaks[2]
\setlist[enumerate]{leftmargin=2.3em,itemsep=0.15em,topsep=0.35em}

\titleformat{\section}
  {\large\bfseries}
  {\thesection.}{0.55em}{}
\titleformat{\subsection}
  {\normalsize\bfseries}
  {\thesubsection}{0.55em}{}
\titlespacing*{\section}{0pt}{1.55em}{0.7em}
\titlespacing*{\subsection}{0pt}{1.15em}{0.45em}

\newtheorem{theorem}{Theorem}[section]
\newtheorem{proposition}[theorem]{Proposition}
\newtheorem{lemma}[theorem]{Lemma}
\newtheorem{corollary}[theorem]{Corollary}
\theoremstyle{definition}
\newtheorem{definition}[theorem]{Definition}
\theoremstyle{remark}
\newtheorem{remark}[theorem]{Remark}

\newcommand{\R}{\mathbb{R}}

\newcommand{\Nrm}[1]{\left\lVert #1\right\rVert}
\newcommand{\ip}[2]{\left\langle #1,#2\right\rangle}
\newcommand{\tr}{\operatorname{tr}}

\begin{document}
\thispagestyle{plain}

% ---------- title block ----------
\begin{center}
  {\fontsize{17}{22}\selectfont\bfseries
   Determinant Factorization of Left Multiplication\\[0.15em]
   in the 32-Dimensional Cayley--Dickson Algebra\par}
  \vspace{0.95em}
  {\normalsize Shoot Koebisu\par}
  \vspace{0.18em}
  {\small Hiroshima University High School, Fukuyama\par}
  \vspace{0.12em}
  {\small \href{mailto:shoot.koe@gmail.com}{shoot.koe@gmail.com}\par}
\end{center}

\vspace{0.8em}
\begin{center}
\begin{minipage}{0.90\linewidth}
\small
\noindent\textbf{Abstract.}\quad
Let $A_5$ denote the 32-dimensional Cayley--Dickson algebra, and let
$L_x(y)=xy$ be left multiplication by $x\in A_5$.  Put
$N(x)=\Nrm{x}^2$ and $B_x=L_{x^*}L_x$, and define
\[
q_k(x)=\frac14\tr(B_x^k)-N(x)^k\qquad (1\le k\le 7).
\]
From these power sums, we construct a homogeneous polynomial $D_{14}$ of degree
$14$ by Newton's identities.  Using the fourfold multiplicity of eigenspaces
and the existence of the eigenvalue $1$ in the eigentheory of
Biss--Christensen--Dugger--Isaksen, we prove that
\[
\det L_x=N(x)^2D_{14}(x)^2
\]
for every $x\in A_5$.  We further show that a nonzero element $x$ is a left
zero divisor if and only if $D_{14}(x)=0$, and that
\[
0\le D_{14}(x)\le N(x)^7.
\]
On the embedded sedenion subalgebra $A_4\subset A_5$, an additional norm
factor occurs in the reduced factor, whereas on $A_5$ itself the norm
polynomial $N$ does not divide $D_{14}$.  The latter statement is proved by an
explicit exact certificate after complexification.  These results provide a
determinant-level partial answer to the problem of studying characteristic
polynomials in higher Cayley--Dickson algebras.
\end{minipage}
\end{center}

\vspace{0.9em}
\section{Introduction}
The Cayley--Dickson construction starts from the real numbers $A_0=\R$ and
recursively doubles the dimension by setting
\[
A_n=A_{n-1}\oplus A_{n-1}.
\]
The algebras $A_1,A_2,A_3$ are isomorphic to the complex numbers, quaternions,
and octonions, respectively; $A_4$ is the 16-dimensional sedenion algebra, and
$A_5$ is the 32-dimensional Cayley--Dickson algebra.  Zero divisors occur for
$n\ge4$, and the degeneracy of left multiplication is therefore a basic
structural problem in these algebras \cite{Moreno1998,BDI2008}.

For $x\in A_5$, let
\[
L_x:A_5\to A_5,\qquad L_x(y)=xy
\]
denote left multiplication.  Since $A_5$ is a 32-dimensional real vector
space, $L_x$ is represented by a $32\times32$ real matrix.  Hence a nonzero
$x$ is a left zero divisor if and only if
\[
\ker L_x\ne0,
\]
or equivalently if and only if $\det L_x=0$.

In the author's previous work \cite{Koebisu2026}, an explicit factorization
for the sedenions was obtained:
\[
\det L_a=N(a)^4D_2(a)^2.
\]
The purpose of the present paper is to determine what form of determinant
factorization persists in the next Cayley--Dickson algebra $A_5$.

The main input is the eigentheory developed by
Biss--Christensen--Dugger--Isaksen \cite{BCDI2009}.  For the normalized
operator
\[
M_x=\frac{1}{N(x)}L_{x^*}L_x,
\]
they prove that every eigenspace has real dimension divisible by four and that,
for $x\ne0$, the eigenvalue $1$ is always present.  In the same paper,
Question~9.5 asks for the study of the characteristic polynomial of elements in
higher Cayley--Dickson algebras \cite[Question~9.5]{BCDI2009}.

The present paper does not determine the full characteristic polynomial.
Instead, it translates the determinant-level spectral information into a
basis-independent polynomial identity.  A technical point is that the
nontrivial eigenvalues need not admit a global continuous labeling as $x$
varies.  We therefore do not define the reduced factor by first choosing and
multiplying eigenvalue branches.  Rather, we first define a global degree-$14$
polynomial $D_{14}$ from trace power sums using Newton's identities, and only
then prove that it agrees with the relevant spectral product.

Our main result is
\[
\det L_x=N(x)^2D_{14}(x)^2.
\]
We also prove that the zero set of $D_{14}$ is precisely the set of nonzero
left zero divisors, establish the bound $0\le D_{14}\le N^7$, describe the
restriction to the sedenion subalgebra, and show that $N\nmid D_{14}$ on
$A_5$.  The last statement gives a structural contrast with $A_4$, where the
corresponding reduced factor contains an additional norm factor.

\section{Cayley--Dickson algebras and left multiplication}
\subsection{Construction and inner product}
We use the Cayley--Dickson convention of \cite{BCDI2009}.  Set $A_0=\R$ and,
recursively,
\[
A_n=A_{n-1}\times A_{n-1}.
\]
For $a,b,c,d\in A_{n-1}$, define
\begin{align}
(a,b)^*&=(a^*,-b),\label{eq:conj}\\
(a,b)(c,d)&=(ac-d^*b,\;da+bc^*).\label{eq:product}
\end{align}
Let $\operatorname{Re}(x)=(x+x^*)/2$ and define the standard real inner
product by
\[
\ip{x}{y}:=\operatorname{Re}(xy^*).
\]
With respect to the standard basis $e_0,\ldots,e_{2^n-1}$,
\[
\ip{x}{y}=\sum_{j=0}^{2^n-1}x_jy_j,
\]
so this is the usual Euclidean inner product.  In particular it is positive
definite, and we write
\[
N(x):=\Nrm{x}^2=\ip{x}{x}=\sum_{j=0}^{2^n-1}x_j^2.
\]
Define the left and right multiplication operators by
\[
L_x(y)=xy,\qquad R_x(y)=yx.
\]

\begin{lemma}[Adjoint formula]\label{lem:adjoint}
For every $x\in A_n$, the adjoint of $L_x$ with respect to the standard real
inner product is
\[
L_x^*=L_{x^*}.
\]
This is \cite[Lemma 2.3]{BCDI2009}.
\end{lemma}

Consequently,
\[
B_x:=L_{x^*}L_x=L_x^*L_x
\]
is a real symmetric positive semidefinite operator.  In particular, all of its
eigenvalues are nonnegative real numbers.  Moreover,
\begin{equation}
\det B_x=\det(L_x^*L_x)=(\det L_x)^2.
\label{eq:detB}
\end{equation}

\subsection{Normalized operator and known eigentheory}
For $x\ne0$, following \cite[Definition 3.1]{BCDI2009}, define
\[
M_x:=\frac{1}{N(x)}B_x
=\frac{1}{N(x)}L_{x^*}L_x.
\]

\begin{theorem}[Fourfold multiplicity of eigenspaces]\label{thm:fourfold}
Let $n\ge2$.  For every $x\in A_n\setminus\{0\}$, each eigenspace of $M_x$ has
real dimension divisible by four \cite[Proposition 3.20]{BCDI2009}.
\end{theorem}

\begin{theorem}[Existence of the eigenvalue $1$]\label{thm:one}
For every $x\in A_n\setminus\{0\}$, the number $1$ is an eigenvalue of $M_x$
\cite[Proposition 3.13]{BCDI2009}.
\end{theorem}

Since $M_x$ is symmetric positive semidefinite, it is diagonalizable and all
of its eigenvalues are nonnegative \cite[Proposition 3.9]{BCDI2009}.

\section{Polynomial construction of $D_{14}$}
In this section we avoid defining $D_{14}$ by locally labeling eigenvalues.
Instead, we construct it directly as a polynomial from trace invariants.

\begin{definition}[Power-sum invariants]\label{def:qk}
For $x\in A_5$, put $B_x=L_{x^*}L_x$.  For $1\le k\le7$, define
\begin{equation}
q_k(x):=\frac14\tr(B_x^k)-N(x)^k.
\label{eq:qk}
\end{equation}
\end{definition}

\begin{lemma}\label{lem:qhomogeneous}
Each $q_k$ is a homogeneous polynomial of degree $2k$ on
$A_5\cong\R^{32}$.
\end{lemma}

\begin{proof}
With respect to any fixed real basis, each entry of $L_x$ is linear in the
coordinates of $x$.  Hence each entry of $B_x=L_{x^*}L_x$ is a homogeneous
quadratic polynomial, and each entry of $B_x^k$, as well as
$\tr(B_x^k)$, is homogeneous of degree $2k$.  The same is true of $N(x)^k$,
so the claim follows from \eqref{eq:qk}.
\end{proof}

\begin{definition}[Newton recursion and $D_{14}$]\label{def:D14}
Set $E_0(x)=1$.  For $1\le m\le7$, define recursively
\begin{equation}
E_m(x)
:=\frac1m\sum_{k=1}^{m}(-1)^{k-1}E_{m-k}(x)q_k(x).
\label{eq:newton}
\end{equation}
Finally, set
\begin{equation}
D_{14}(x):=E_7(x).
\label{eq:D14def}
\end{equation}
\end{definition}

\begin{proposition}\label{prop:Dpoly}
The polynomial $E_m$ is homogeneous of degree $2m$.  In particular,
$D_{14}=E_7$ is homogeneous of degree $14$.
\end{proposition}

\begin{proof}
We argue by induction on $m$.  The statement is clear for $E_0=1$.  Assume it
holds for $E_0,\ldots,E_{m-1}$.  Each term $E_{m-k}q_k$ in
\eqref{eq:newton} has degree
\[
2(m-k)+2k=2m.
\]
Hence $E_m$ is homogeneous of degree $2m$.
\end{proof}

\begin{proposition}[First power sum]\label{prop:q1}
For every $x\in A_5$,
\[
\tr(B_x)=32N(x),\qquad q_1(x)=7N(x).
\]
\end{proposition}

\begin{proof}
Substituting $y=x$ into \cite[Lemma~3.16]{BCDI2009} gives
\[
\tr(L_{x^*}L_x)=2^5\ip{x}{x}=32N(x).
\]
Therefore, by Definition~\ref{def:qk},
\[
q_1=\frac14\tr(B_x)-N=8N-N=7N.
\]
\end{proof}

\section{Spectral interpretation}
We now show that the polynomial $D_{14}$ constructed above has the expected
interpretation as a product of spectral blocks.

\begin{lemma}[Fourfold block representation]\label{lem:block}
Let $x\in A_5\setminus\{0\}$.  Then there exist nonnegative real numbers
\[
\beta_1(x),\ldots,\beta_8(x),
\]
listed with repetitions allowed, such that the characteristic polynomial of
$B_x$ is
\begin{equation}
\chi_{B_x}(t)=\prod_{j=1}^{8}(t-\beta_j(x))^4.
\label{eq:charblock}
\end{equation}
Moreover, after relabeling, one may choose
\begin{equation}
\beta_1(x)=N(x).
\label{eq:beta1}
\end{equation}
\end{lemma}

\begin{proof}
Since $B_x=N(x)M_x$, Theorem~\ref{thm:fourfold} implies that every eigenvalue
of $B_x$ has multiplicity divisible by four.  Because
$\dim_{\R}A_5=32=4\cdot8$, listing each eigenvalue as many times as its
multiplicity divided by four yields eight numbers $\beta_j$ satisfying
\eqref{eq:charblock}.

By Theorem~\ref{thm:one}, $1$ is an eigenvalue of $M_x$.  Hence $N(x)$ is an
eigenvalue of $B_x$ with multiplicity at least four, so one fourfold block may
be chosen as $\beta_1=N(x)$.
\end{proof}

\begin{remark}
The individual functions $\beta_j(x)$ need not admit a unique or globally
continuous labeling as $x$ varies.  This is precisely why $D_{14}$ was defined
first from the symmetric trace invariants $q_k$, rather than from a chosen
ordering of the $\beta_j$.
\end{remark}

\begin{proposition}[Spectral interpretation of the Newton construction]\label{prop:spectralD}
Let $x\in A_5\setminus\{0\}$, and choose $\beta_1,\ldots,\beta_8$ as in
Lemma~\ref{lem:block} with $\beta_1=N(x)$.  Then
\begin{equation}
D_{14}(x)=\prod_{j=2}^{8}\beta_j(x).
\label{eq:Dprod}
\end{equation}
In particular, $D_{14}(x)\ge0$.
\end{proposition}

\begin{proof}
From \eqref{eq:charblock}, for every $k\ge1$,
\[
\tr(B_x^k)=4\sum_{j=1}^{8}\beta_j^k.
\]
Using $\beta_1=N(x)$ gives
\[
q_k(x)=\frac14\tr(B_x^k)-N(x)^k
=\sum_{j=2}^{8}\beta_j^k.
\]
Thus $q_1,\ldots,q_7$ are the power sums of the seven numbers
$\beta_2,\ldots,\beta_8$.  Newton's identities \eqref{eq:newton} show that
$E_m$ is the $m$th elementary symmetric polynomial in these seven numbers.
In particular,
\[
D_{14}(x)=E_7(x)=\beta_2\beta_3\cdots\beta_8.
\]
Since each $\beta_j$ is an eigenvalue of the positive semidefinite operator
$B_x$, all are nonnegative.
\end{proof}

\section{Determinant factorization}
\subsection{The square structure}
\begin{proposition}\label{prop:detBfactor}
For every $x\in A_5$,
\begin{equation}
\det B_x=\bigl(N(x)D_{14}(x)\bigr)^4.
\label{eq:detBfactor}
\end{equation}
\end{proposition}

\begin{proof}
If $x=0$, both sides vanish.  Suppose $x\ne0$.  By
Lemma~\ref{lem:block} and Proposition~\ref{prop:spectralD},
\begin{align*}
\det B_x
&=\prod_{j=1}^{8}\beta_j^4
 =\left(\prod_{j=1}^{8}\beta_j\right)^4\\
&=\left(N(x)\prod_{j=2}^{8}\beta_j\right)^4
 =\bigl(N(x)D_{14}(x)\bigr)^4.
\end{align*}
\end{proof}

Combining \eqref{eq:detBfactor} with \eqref{eq:detB} gives
\begin{equation}
(\det L_x)^2=\bigl(N(x)D_{14}(x)\bigr)^4.
\label{eq:absfactor}
\end{equation}
Thus
\[
|\det L_x|=N(x)^2D_{14}(x)^2.
\]
The next lemma removes the sign ambiguity.

\subsection{Nonnegativity of the left-multiplication determinant}
\begin{lemma}\label{lem:nonnegative}
For every $n\ge1$ and every $x\in A_n$,
\[
\det L_x\ge0.
\]
\end{lemma}

\begin{proof}
Write $x=r1+u$, where $r=\operatorname{Re}(x)\in\R$ and $u^*=-u$.
By linearity of left multiplication,
\[
L_x=rI+L_u.
\]
Lemma~\ref{lem:adjoint} gives
\[
L_u^*=L_{u^*}=-L_u,
\]
so $L_u$ is a real skew-symmetric operator.  By the real canonical form of a
skew-symmetric matrix, there exist an orthogonal matrix $Q$ and numbers
$\sigma_j\ge0$ such that
\[
Q^{\mathsf T}L_uQ
=\bigoplus_{j=1}^{2^{n-1}}
\begin{pmatrix}
0&-\sigma_j\\
\sigma_j&0
\end{pmatrix}.
\]
Here we used the fact that $\dim_{\R}A_n=2^n$ is even.  Therefore
\begin{align*}
\det L_x
&=\det(rI+L_u)\\
&=\prod_{j=1}^{2^{n-1}}
\det\begin{pmatrix}
r&-\sigma_j\\
\sigma_j&r
\end{pmatrix}\\
&=\prod_{j=1}^{2^{n-1}}(r^2+\sigma_j^2)\ge0.
\end{align*}
\end{proof}

\begin{theorem}[Main theorem]\label{thm:main}
Let $A_5$ be the 32-dimensional Cayley--Dickson algebra.  For the homogeneous
polynomial $D_{14}$ of degree $14$ defined in Definition~\ref{def:D14}, one has
\begin{equation}
\det L_x=N(x)^2D_{14}(x)^2
\label{eq:main}
\end{equation}
for every $x\in A_5$.  In coordinates $x=(x_0,\ldots,x_{31})$,
\[
N(x)=x_0^2+x_1^2+\cdots+x_{31}^2.
\]
\end{theorem}

\begin{proof}
Equation \eqref{eq:absfactor} gives
\[
|\det L_x|=N(x)^2D_{14}(x)^2.
\]
By Lemma~\ref{lem:nonnegative}, $\det L_x\ge0$, so the absolute value may be
removed.
\end{proof}

The degrees are consistent:
\[
\deg N^2=4,\qquad \deg D_{14}^2=28,
\]
and therefore
\[
4+28=32=\deg(\det L_x).
\]

\section{Basic bounds and equality}
The trace identity in Proposition~\ref{prop:q1}, together with the spectral
interpretation, yields an immediate useful bound for $D_{14}$.

\begin{proposition}[Upper bound]\label{prop:bound}
For every $x\in A_5$,
\begin{equation}
0\le D_{14}(x)\le N(x)^7.
\label{eq:D-bound}
\end{equation}
Consequently,
\begin{equation}
0\le \det L_x\le N(x)^{16}=\Nrm{x}^{32}.
\label{eq:det-bound}
\end{equation}
Moreover, $D_{14}(1)=1$.
\end{proposition}

\begin{proof}
The case $x=0$ is immediate.  Suppose $x\ne0$ and choose the fourfold blocks
$\beta_1=N,\beta_2,\ldots,\beta_8$ as in Lemma~\ref{lem:block}.  By
Proposition~\ref{prop:q1},
\[
\sum_{j=1}^{8}\beta_j=\frac14\tr(B_x)=8N,
\]
so
\[
\sum_{j=2}^{8}\beta_j=7N.
\]
Since all $\beta_j\ge0$, the arithmetic--geometric mean inequality gives
\[
D_{14}=\prod_{j=2}^{8}\beta_j
\le\left(\frac{\beta_2+\cdots+\beta_8}{7}\right)^7=N^7.
\]
Nonnegativity follows from Proposition~\ref{prop:spectralD}, and
\eqref{eq:det-bound} then follows from the main theorem.  For $x=1$, one has
$B_1=I$, so every $\beta_j$ equals $1$ and hence $D_{14}(1)=1$.
\end{proof}

\begin{proposition}[Equality case]\label{prop:equality}
Let $x\ne0$.  Then
\[
D_{14}(x)=N(x)^7
\]
if and only if $M_x=I$.  Hence, in the sense of \cite[Definition~4.1 and the
discussion following it]{BCDI2009}, this is equivalent to $x$ being an
alternative element.
\end{proposition}

\begin{proof}
Equality in the arithmetic--geometric mean inequality occurs if and only if
\[
\beta_2=\cdots=\beta_8=N.
\]
Since $\beta_1=N$ as well, this is equivalent to $B_x=NI$, or equivalently
$M_x=I$.  The final equivalence is the characterization of alternative
elements in \cite{BCDI2009} via the extremal eigenvalues.
\end{proof}

\section{Zero divisors}
\begin{corollary}[Polynomial criterion for left zero divisors]\label{cor:zd}
Let $x\in A_5\setminus\{0\}$.  Then
\begin{equation}
\boxed{
 x\text{ is a left zero divisor}\quad\Longleftrightarrow\quad D_{14}(x)=0
}
\label{eq:zd}
\end{equation}
\end{corollary}

\begin{proof}
For the finite-dimensional linear map $L_x$, the element $x$ is a left zero
divisor if and only if $\ker L_x\ne0$, equivalently $\det L_x=0$.  Since
$x\ne0$ implies $N(x)>0$, the main theorem gives
\[
\det L_x=0
\iff D_{14}(x)^2=0
\iff D_{14}(x)=0.
\]
\end{proof}

\begin{remark}
It is also shown in \cite[Lemma 3.8]{BCDI2009} that
$\ker M_x=\ker L_x$.  From Proposition~\ref{prop:spectralD}, the condition
$D_{14}(x)=0$ means that at least one of the nontrivial fourfold spectral
blocks has eigenvalue zero.
\end{remark}

\section{Comparison with the sedenion subalgebra}
We identify $A_4$ with the subalgebra $A_4\times\{0\}\subset A_5$.  For
$a\in A_4$, write $\widetilde a=(a,0)\in A_5$.

\begin{remark}[Meaning of the subscripts]
The subscript $14$ in $D_{14}$ denotes the degree of the polynomial.  In the
author's previous work \cite{Koebisu2026}, by contrast, the subscripts in
$D_1,D_2$ enumerate factors and have a different meaning.
\end{remark}

\begin{lemma}\label{lem:restriction}
\[
\det L_{\widetilde a}^{A_5}=\bigl(\det L_a^{A_4}\bigr)^2.
\]
\end{lemma}

\begin{proof}
By \eqref{eq:product},
\[
(a,0)(c,d)=(ac,da).
\]
Hence, with respect to $A_5=A_4\oplus A_4$,
\[
L_{\widetilde a}^{A_5}
=\begin{pmatrix}
L_a^{A_4}&0\\
0&R_a^{A_4}
\end{pmatrix}.
\]
Let $C(y)=y^*$ denote conjugation.  Then $CR_aC=L_{a^*}$, and by
Lemma~\ref{lem:adjoint}, $L_{a^*}=L_a^*$.  Thus
$\det R_a=\det L_a$, proving the claim.
\end{proof}

Using the notation of the author's previous work \cite{Koebisu2026}, write
\[
a=v_1+v_2e_8,\qquad
v_1=x_1e_0+u,\qquad
v_2=x_2e_0+w,
\qquad u,w\in\operatorname{Im}(\mathbb O).
\]
Then
\begin{align}
D_2(a)
&=N(a)^2
 -4\bigl(\Nrm{u}^2\Nrm{w}^2-\ip{u}{w}^2\bigr),
\label{eq:sedenionD2}\\
\det L_a^{A_4}
&=N(a)^4D_2(a)^2.
\label{eq:sedeniondet}
\end{align}
Therefore, Lemma~\ref{lem:restriction} and the main theorem imply
\begin{equation}
D_{14}(a,0)
=N(a)^3
\left[
N(a)^2-4\bigl(\Nrm{u}^2\Nrm{w}^2-\ip{u}{w}^2\bigr)
\right]^2.
\label{eq:restriction-explicit}
\end{equation}

The difference becomes clearer when comparing the reduced factors in the
general construction.  Let $E_3^{(4)}$ denote the degree-$6$ polynomial
obtained by applying Theorem~\ref{thm:general} to $A_4$.  Then
\[
\det L_a^{A_4}=N(a)^2\bigl(E_3^{(4)}(a)\bigr)^2.
\]
Comparing this identity with \eqref{eq:sedeniondet}, and using the
nonnegativity of $E_3^{(4)}$ as a spectral product together with the
nonnegativity of $D_2$ proved in \cite{Koebisu2026}, gives
\begin{equation}
E_3^{(4)}(a)=N(a)D_2(a).
\label{eq:A4-extra-N}
\end{equation}
Thus, in $A_4$, the reduced factor itself always contains $N$ as a factor.
This is consistent with \cite[Corollary~7.3]{BCDI2009}, where under the stated
generic hypotheses the eigenvalue $1$ has multiplicity $8$, corresponding to
two fourfold blocks.  In degenerate limits the eigenvalues may merge and the
multiplicity may increase further.

\subsection{The norm does not divide $D_{14}$ on $A_5$}
In contrast with \eqref{eq:A4-extra-N}, no norm factor can be factored from
$D_{14}$ in general on $A_5$.  We prove this by an exact one-point evaluation.

Complexify both $A_5$ and the polynomials $N,D_{14}$, and write
$A_{5,\mathbb C}:=A_5\otimes_{\R}\mathbb C$.

\begin{proposition}[Absence of a norm factor]\label{prop:N-not-divide}
\[
N\nmid D_{14}
\qquad\text{in }\mathbb R[x_0,\ldots,x_{31}].
\]
\end{proposition}

\begin{proof}
In $A_{5,\mathbb C}$, set
\begin{equation}
z=e_4+e_{10}-e_{23}
+i(e_{11}-e_{16}+e_{18}),
\label{eq:certificate-z}
\end{equation}
where $i^2=-1$ is the imaginary unit of the coefficient field $\mathbb C$ and
is distinct from the Cayley--Dickson basis elements.  Orthogonality of the
standard basis gives
\[
N(z)=1+1+1+i^2+i^2+i^2=0.
\]

Extend \eqref{eq:conj}--\eqref{eq:product} $\mathbb C$-linearly and compute
$B_z=L_{z^*}L_z$ exactly.  The power sums from Definition~\ref{def:qk} are
\begin{equation}
(q_1,q_2,q_3,q_4,q_5,q_6,q_7)
=(0,-16,-96,-64,-640,-1024,14336).
\label{eq:q-certificate}
\end{equation}
Newton's recursion \eqref{eq:newton} yields
\[
(E_0,E_1,E_2,E_3,E_4,E_5,E_6,E_7)
=(1,0,8,-32,48,-384,896,-512).
\]
Therefore
\[
D_{14}(z)=E_7(z)=-512\ne0.
\]
If $N$ divided $D_{14}$ in $\mathbb R[x_0,\ldots,x_{31}]$, then after
complexification we would still have $D_{14}=NQ$ for some polynomial $Q$.
The equality $N(z)=0$ would then force $D_{14}(z)=0$, a contradiction.
\end{proof}

\begin{remark}
Proposition~\ref{prop:N-not-divide} does not assert that $D_{14}$ is
irreducible.  It shows only that the quadratic norm polynomial $N$ is not a
factor of $D_{14}$.
\end{remark}

\section{A general consequence in dimension $2^n$}
The proof of the main theorem uses almost no computation specific to $A_5$.
The same argument applies to $A_n$ once one uses the fourfold multiplicity of
eigenspaces and the existence of the eigenvalue $1$.

\begin{theorem}[General form]\label{thm:general}
Let $n\ge2$ and put $m=2^{n-2}-1$.  For $x\in A_n$, set
\[
B_x=L_{x^*}L_x,\qquad N(x)=\Nrm{x}^2,
\]
and define
\[
q_k^{(n)}(x)=\frac14\tr(B_x^k)-N(x)^k
\qquad (1\le k\le m).
\]
Let $E_m^{(n)}$ be constructed from these power sums by Newton's recursion.
Then $E_m^{(n)}$ is a homogeneous polynomial of degree
\[
2m=2^{n-1}-2,
\]
and
\begin{equation}
\det L_x=N(x)^2\bigl(E_m^{(n)}(x)\bigr)^2.
\label{eq:general}
\end{equation}
\end{theorem}

\begin{proof}
By Theorem~\ref{thm:fourfold}, the $2^n$ eigenvalues of $B_x$ can be counted as
$2^{n-2}$ fourfold blocks.  By Theorem~\ref{thm:one}, one block may be chosen
to be $N(x)$.  The remaining $m=2^{n-2}-1$ blocks have power sums
$q_k^{(n)}$, so Newton's identities show that $E_m^{(n)}$ is their product.
Hence
\[
\det B_x=\bigl(N(x)E_m^{(n)}(x)\bigr)^4.
\]
Using \eqref{eq:detB} and Lemma~\ref{lem:nonnegative} gives
\eqref{eq:general}.  The degree statement follows by the same induction as in
Proposition~\ref{prop:Dpoly}.
\end{proof}

For $n=5$, one has $m=7$ and $E_7^{(5)}=D_{14}$.  For $n=2$, one has $m=0$;
interpreting the empty product as $E_0^{(2)}=1$ gives
\[
\det L_x=N(x)^2,
\]
as expected for the quaternions.

\section{Discussion}
The main theorem translates the known fourfold eigenspace structure into a
polynomial identity at the determinant level.  Biss--Christensen--Dugger--Isaksen
ask in Question~9.5 for the study of characteristic polynomials in higher
Cayley--Dickson algebras \cite[Question~9.5]{BCDI2009}.  The present work does
not determine the full characteristic polynomial, but it gives a partial
answer at the level of its constant term by constructing the basis-independent
degree-$14$ polynomial $D_{14}$.

The comparison between $A_4$ and $A_5$ is particularly significant.  In
$A_4$, equation \eqref{eq:A4-extra-N} shows that the reduced factor
$E_3^{(4)}$ contains an additional factor $N$.  Proposition~\ref{prop:N-not-divide}
shows that this no longer holds in $A_5$.  Thus, although the general
factorization in Theorem~\ref{thm:general} can refine further in lower
dimensions, at least the same additional norm factor does not persist in
$A_5$.

Proposition~\ref{prop:bound} also gives the sharp bounds
\[
0\le D_{14}\le N^7,
\qquad
0\le\det L_x\le\Nrm{x}^{32}.
\]
The equality case corresponds to $M_x=I$, equivalently to alternative
elements.  This suggests that $D_{14}$ may be interpreted not only as a
zero-divisor detector, but also as a measure of the deviation of left
multiplication from the norm-multiplicative case.

The following questions remain open.
\begin{enumerate}
  \item Does $D_{14}$ admit any further nontrivial factorization over
  $\mathbb R[x_0,\ldots,x_{31}]$ or $\mathbb Q[x_0,\ldots,x_{31}]$?
  Proposition~\ref{prop:N-not-divide} resolves only the possibility of $N$ as
  a factor.
  \item Can $D_{14}$ be expressed using a small collection of geometric
  invariants, such as norms, inner products, or associators?
  \item Can the coefficient ring of $D_{14}$ be determined more precisely?
  The definition involves the factor $1/4$ and denominators in Newton's
  recursion, so the question whether
  $D_{14}\in\mathbb Z[x_0,\ldots,x_{31}]$ requires separate analysis.
  \item How can the strata of the zero-divisor variety $D_{14}=0$ be
  described according to $\dim\ker L_x$?  The fourfold multiplicity implies
  that the kernel dimension is divisible by four, and in $A_5$ the general
  upper bound is $16$ \cite{BDI2008}.
\end{enumerate}

Question~9.2 of \cite{BCDI2009} reports computer evidence for an element of
$A_5$ for which the eigenvalue $1$ has multiplicity $4$.  This supports a
contrast with the generic multiplicity-$8$ structure appearing in
\cite[Corollary~7.3]{BCDI2009} for $A_4$, but we do not claim that multiplicity
$4$ occurs for every generic element of $A_5$.

\section{Conclusion}
For the 32-dimensional Cayley--Dickson algebra $A_5$, we defined a homogeneous
polynomial $D_{14}$ of degree $14$ from the trace power sums of
\[
B_x=L_{x^*}L_x
\]
via Newton's identities.  Using the known fourfold multiplicity of eigenspaces
and the existence of the eigenvalue $1$, we proved that
\[
\det L_x=N(x)^2D_{14}(x)^2
\]
for every $x\in A_5$.  For nonzero elements, we further obtained
\[
x\text{ is a left zero divisor}\iff D_{14}(x)=0.
\]

We also proved
\[
0\le D_{14}(x)\le N(x)^7,
\]
with equality characterized by alternative elements.  On the embedded
sedenion subalgebra, the reduced factor contains an additional norm factor,
whereas an explicit exact certificate after complexification shows that
\[
N\nmid D_{14}
\]
on $A_5$.  This provides a concrete change in determinant-factorization
structure from $A_4$ to $A_5$.

The construction extends to general $A_n$, but the further factorization,
integrality, geometric interpretation, and zero-divisor stratification of
$D_{14}$ remain open problems.

\appendix
\section{Explicit Newton formula for $D_{14}$}
In terms of the invariants $q_1,\ldots,q_7$ from Definition~\ref{def:qk},
Newton's identities give
\begin{align}
D_{14}=\frac1{5040}\Bigl(&
q_1^7-21q_1^5q_2+70q_1^4q_3+105q_1^3q_2^2
-210q_1^3q_4\notag\\
&-420q_1^2q_2q_3+504q_1^2q_5-105q_1q_2^3
+630q_1q_2q_4\notag\\
&+280q_1q_3^2-840q_1q_6+210q_2^2q_3-504q_2q_5
-420q_3q_4+720q_7
\Bigr).
\label{eq:explicitD14}
\end{align}
Every term on the right has total degree $14$, because $q_k$ has degree
$2k$.

Substituting $q_1=7N$ from Proposition~\ref{prop:q1}, one obtains a formula
using only $N,q_2,\ldots,q_7$:
\begin{align}
D_{14}={}&
\frac{117649}{720}N^7
-\frac{16807}{240}N^5q_2
+\frac{2401}{72}N^4q_3
+\frac{343}{48}N^3q_2^2
-\frac{343}{24}N^3q_4\notag\\
&-\frac{49}{12}N^2q_2q_3
+\frac{49}{10}N^2q_5
-\frac{7}{48}Nq_2^3
+\frac{7}{8}Nq_2q_4
+\frac{7}{18}Nq_3^2
-\frac{7}{6}Nq_6\notag\\
&+\frac1{24}q_2^2q_3
-\frac1{10}q_2q_5
-\frac1{12}q_3q_4
+\frac17q_7.
\label{eq:explicitD14-simplified}
\end{align}
Formula \eqref{eq:explicitD14-simplified} reduces the independent trace
invariants required for explicit computation to $q_2,\ldots,q_7$.

\section{Exact certificate for Proposition~\ref{prop:N-not-divide}}
Let $z$ be the element defined in \eqref{eq:certificate-z}.  Since $N(z)=0$,
\eqref{eq:q-certificate} gives directly
\[
\frac14\tr(B_z^k)=q_k(z)
\qquad(1\le k\le7).
\]
Hence
\[
(\tr B_z,\tr B_z^2,\ldots,\tr B_z^7)
=(0,-64,-384,-256,-2560,-4096,57344).
\]
These values are obtained by exact matrix multiplication and traces of the
$32\times32$ matrix $B_z$ with Gaussian-integer coefficients.  Applying
Newton's recursion successively yields
\begin{align*}
E_1&=0, & E_2&=8, & E_3&=-32, & E_4&=48,\\
E_5&=-384, & E_6&=896, & E_7&=-512.
\end{align*}
Thus $D_{14}(z)=-512\ne0$.  This appendix supplies a finite exact certificate
over the integers and Gaussian integers, rather than a numerical
approximation.

\end{document}